\documentclass[11pt, notitlepage]{article}
\usepackage{amssymb,amsmath,comment}
\usepackage{amsthm}
\catcode`\@=11 \@addtoreset{equation}{section}
\def\thesection{\arabic{section}}

\def\theequation{\thesection.\arabic{equation}}
\catcode`\@=12
\usepackage{colortbl}
\usepackage{hyperref}
\usepackage[mathscr]{eucal}
\usepackage{epsf}
\usepackage{esint}
\usepackage{a4wide}

\newcommand{\Om} {\Omega}

\newcommand{\noi} {\noindent}
\newcommand{\na} {\nabla}

\usepackage[all]{xy}
\catcode`\@=11
\def\theequation{\@arabic{\c@section}.\@arabic{\c@equation}}
\catcode`\@=12

\newtheorem{Theorem}{Theorem}[section]
\newtheorem{Lemma}[Theorem]{Lemma}

\newtheorem{Corollary}[Theorem]{Corollary}
\newtheorem{Remark}[Theorem]{Remark}
\newtheorem{Definition}[Theorem]{Definition}

\begin{document}

{\vspace{0.01in}}

\title{On mixed local–nonlocal Sobolev-type inequalities and their connection with singular equations in the Heisenberg group}

\author{Prashanta Garain}

\maketitle

\begin{abstract}\noindent
In this work, we establish a mixed local--nonlocal Sobolev-type inequality in the Heisenberg group and demonstrate that its extremals coincide with solutions to the corresponding mixed local--nonlocal singular $p$-Laplace equations. We further show that these inequalities serve as a necessary and sufficient condition for the existence of weak solutions to the associated singular problems. Notably, the same characterization remains valid in both the purely local and purely nonlocal settings. Our results thus provide a unified framework linking the existence theory for singular equations across local, nonlocal, and mixed regimes.
\end{abstract}

\maketitle

\noi {Keywords: Sobolev type inequality, extremal, mixed local-nonlocal singular problem, Heisenberg group.}

\noi{\textit{2020 Mathematics Subject Classification: 35A23, 35H20, 35J92, 35R11,  35J75.}

\bigskip

\tableofcontents

\section{Introduction}
In the Euclidean setting, Sobolev-type inequalities connected to singular elliptic problems have been recently investigated in the local~\cite{GFA}, nonlocal~\cite{Nmn}, and mixed local--nonlocal frameworks~\cite{GU21}. More precisely, in~\cite{GFA} the authors proved that, for a bounded smooth domain $\Omega \subset \mathbb{R}^N$ and parameters $0<\delta<1<p<\infty$, the quantity
\begin{equation}\label{nu-def}
\mu(\Omega)
:= 
\inf_{v \in W^{1,p}_0(\Omega)\setminus\{0\}}
\left\{
\int_{\Omega} |\nabla v|^{p}\,dx
\;:\;
\int_{\Omega} |v|^{1-\delta}\,dx = 1
\right\}
\tag{1.4}
\end{equation}
is attained by some $u \in W^{1,p}_0(\Omega)$, which is a positive solution to the singular $p$-Laplace equation
\[
-\Delta_p u:=\text{div}(|\na u|^{p-2}\na u) = \mu(\Omega)\,u^{-\delta} \quad \text{in } \Omega,
\qquad
u>0 \ \text{in } \Omega,
\qquad
u=0 \ \text{on } \partial\Omega.
\]
The corresponding anisotropic setting has been addressed in~\cite{BGmm22, BGdie}. In the nonlocal framework, the authors of~\cite{Nmn} established that for any $0<\delta<1$, $0<s<1<p<\infty$, and for a given nonnegative function $f \in L^{m}(\Omega)\setminus\{0\}$ with $m \ge 1$, the quantity
\begin{equation}\label{zeta-def}
\nu(\Omega)
:=
\inf_{v\in W^{s,p}_0(\Omega)\setminus\{0\}}
\left\{
\int_{\mathbb{R}^N}\!\int_{\mathbb{R}^N}
\frac{|v(x)-v(y)|^{p}}{|x-y|^{N+ps}}\,dx\,dy
\;:\;
\int_{\Omega} |v|^{1-\delta} f\, dx = 1
\right\}
\end{equation}
is achieved at some $u \in W^{s,p}_0(\Omega)$, which satisfies the singular fractional $p$-Laplace equation
\begin{equation}\label{frac-sing-eq}
(-\Delta_p)^s u = \nu(\Omega)\, f(x)\, u^{-\delta}
\quad \text{in } \Omega, \qquad
u>0 \ \text{in } \Omega, \qquad
u=0 \ \text{in } \mathbb{R}^N \setminus \Omega.
\end{equation}
Here, $(-\Delta_p)^s$ denotes the fractional $p$-Laplace operator, defined by
\[
(-\Delta_p)^s u(x)
:= \mathrm{P.V.}\int_{\mathbb{R}^N}
\frac{|u(x)-u(y)|^{p-2}\big(u(x)-u(y)\big)}{|x-y|^{N+ps}}\,dy,
\]
where P.V. stands for the principal value.
In a recent development, the authors in~\cite{GU21} proved that the infimum
\begin{equation}\label{theta-def}
\theta(\Omega)
:= \inf_{v \in W_0^{1,p}(\Omega) \setminus \{0\}}
\left\{
\int_{\Omega} |\nabla v|^{p}\,dx
\;+\;
\int_{\mathbb{R}^N}\!\int_{\mathbb{R}^N}
\frac{|v(x)-v(y)|^{p}}{|x-y|^{N+sp}}\,dx\,dy
\,:\,
\int_{\Omega} |v|^{1-\delta} f\,dx = 1
\right\},
\end{equation}
is attained by a function $u$ solving the mixed local--nonlocal $p$-Laplace equation
\begin{equation}\label{Meq}
-\Delta_p u + (-\Delta_p)^s u
= \theta(\Omega)\, f(x)\, u^{-\delta}
\quad \text{in } \Omega, \qquad
u>0 \ \text{in } \Omega, \qquad
u=0 \ \text{in } \mathbb{R}^N\setminus\Omega.
\end{equation}
Furthermore, the mixed anisotropic setting has been addressed in~\cite{Gjms}. We also note that singular elliptic problems have been extensively studied over the past decades, resulting in a vast body of literature. For the purely local case, we refer to~\cite{Boc-Or, Canino, CRT}, while the purely nonlocal case is treated in~\cite{Nsop, Caninoetal, EP} and the references therein. More recently, mixed local--nonlocal singular problems have been investigated in~\cite{ARad, Vecchi, Garainjga, GU21} among others.

In non-Euclidean settings, Sobolev-type inequalities of this nature have only recently begun to be understood. For the purely local case in Carnot groups, see~\cite{GUaamp}, and for the purely nonlocal case, we refer to~\cite{GarainHein}. However, to the best of our knowledge, the mixed local--nonlocal case remains unexplored. The primary objective of this article is to bridge this gap.

More precisely, we establish that for $\alpha>0$ and a given nonnegative function $f \in L^{m}(\Omega)\setminus\{0\}$ with $m \ge 1$, the quantity
\begin{equation}\label{Theta}
\Theta(\Omega)
:= \inf_{v \in HW_0^{1,p}(\Omega)\setminus\{0\}}
\left\{
\int_{\Omega} |\nabla_H v|^{p}\,dx
+ 
\alpha \int_{\mathbb{H}^N}\!\int_{\mathbb{H}^N}
\frac{|v(x)-v(y)|^{p}}{|y^{-1}\circ x|^{Q+sp}}\,dx\,dy
:\;
\int_{\Omega} |v|^{1-\delta} f\,dx = 1
\right\},
\end{equation}
is associated to a function $u \in HW_0^{1,p}(\Omega)$ solving the mixed local--nonlocal singular problem
\begin{equation}\label{meqn}
M_\alpha u := -\Delta_{H,p} u + \alpha\, (-\Delta_{H,p})^s u
= f(x)\, u^{-\delta}
\quad \text{in } \Omega,
\qquad
u>0 \ \text{in } \Omega,
\qquad
u=0 \ \text{in } \mathbb{H}^N \setminus \Omega,
\end{equation}
where $0<\delta<1 $ with $0<s<1<p<Q$ and $Q = 2N+2$ denotes the homogeneous dimension of the Heisenberg group $\mathbb{H}^N$. Here,
$$
\Delta_{H,p} u=\text{div}(|\na_H u|^{p-2}\na_H u)
$$
is the $p$-subLaplace operator and
\[
(-\Delta_{H,p})^s u(x)
:= \mathrm{P.V.}\int_{\mathbb{H}^N}
\frac{|u(x)-u(y)|^{p-2}\big(u(x)-u(y)\big)}{|y^{-1}\circ x|^{Q+ps}}\,dy
\]
is the fractional $p$-subLaplace operator.

Moreover, we prove that the mixed local--nonlocal Sobolev-type inequality is both necessary and sufficient for the existence of solutions to~\eqref{meqn}. Analogous characterizations hold separately in the purely local and purely nonlocal cases. As an immediate consequence, we deduce that the purely singular $p$-Laplace equation
\begin{equation}\label{plh}
-\Delta_{H,p} u = f(x)\, u^{-\delta}
\quad \text{in } \Omega,
\qquad
u>0 \ \text{in } \Omega,
\qquad
u=0 \ \text{on } \partial\Omega,
\end{equation}
admits a solution if and only if the mixed local--nonlocal singular $p$-Laplace equation~\eqref{meqn} has a solution for every $\alpha>0$.

To establish our main results, we adopt the methodologies developed in~\cite{GFA, Nmn, GU21}. This necessitates a detailed analysis of the corresponding singular mixed problem. We highlight that mixed local--nonlocal problems in non-Euclidean settings remain largely unexplored, even in the nonsingular case. For recent progress in this direction, we refer to~\cite{DC, Zhang, Zhang2}, where existence and regularity results for nonsingular mixed equations were obtained. Very recently, singular mixed problems in the Heisenberg group have been investigated in~\cite{G1}.

\subsection{Functional setting}
In this subsection, we briefly recall some fundamental properties of the Heisenberg group $\mathbb{H}^N$ and introduce the function spaces relevant to our analysis.

The Euclidean space $\mathbb{R}^{2N+1}$, equipped with the group law
\[
\xi \circ \eta =
\left(
x_1 + y_1,\,
x_2 + y_2,\,
\ldots,\,
x_{2N} + y_{2N},\,
\tau + \tau'
+ \frac{1}{2}\sum_{i=1}^{N}\big( x_i y_{N+i} - x_{N+i} y_i \big)
\right),
\]
where $\xi = (x_1,\ldots,x_{2N},\tau)$ and $\eta = (y_1,\ldots,y_{2N},\tau') \in \mathbb{R}^{2N+1}$, constitutes the Heisenberg group $\mathbb{H}^N$.

The left-invariant vector fields on $\mathbb{H}^N$ are given by
\[
X_i = \partial_{x_i} - \frac{x_{N+i}}{2}\,\partial_\tau,
\qquad
X_{N+i} = \partial_{x_{N+i}} + \frac{x_i}{2}\,\partial_\tau,
\qquad 1 \le i \le N,
\]
and their only nontrivial commutator is
\[
T = \partial_\tau = [X_i, X_{N+i}]
= X_i X_{N+i} - X_{N+i} X_i,
\qquad 1 \le i \le N.
\]
We refer to $X_1, X_2, \ldots, X_{2N}$ as the \emph{horizontal} vector fields on $\mathbb{H}^N$, and to $T$ as the \emph{vertical} vector field.

The Haar measure on $\mathbb{H}^N$ coincides with the Lebesgue measure on $\mathbb{R}^{2N+1}$. For any measurable set $E \subset \mathbb{H}^N$, its Lebesgue measure will be denoted by $|E|$.

For a point $\xi = (x_1,\ldots,x_{2N},\tau)$, we introduce the Kor\'anyi-type norm
\[
|\xi|
= \left( \left( \sum_{i=1}^{2N} x_i^2 \right)^{2} + \tau^{2} \right)^{1/4}.
\]
The Carnot--Carath\'eodory distance between two points $\xi,\eta \in \mathbb{H}^N$ is defined as the infimum of the lengths of horizontal curves joining them and is denoted by $d(\xi,\eta)$. This distance is equivalent to the Kor\'anyi metric:
\[
d(\xi,\eta) \sim |\xi^{-1} \circ \eta|.
\]

The ball of radius $r>0$ centered at $\xi_0$, with respect to the distance $d$, is given by
\[
B_r(\xi_0)
= \{\xi \in \mathbb{H}^N : d(\xi,\xi_0) < r\}.
\]
When the center is unimportant or clear from context, we simply write $B_r := B_r(\xi_0)$.

The homogeneous dimension of $\mathbb{H}^N$ is $Q = 2N + 2$.  
For $1 \le p < \infty$ and an open set $\Omega \subset \mathbb{H}^N$, the Sobolev space
\[
HW^{1,p}(\Omega)
=
\left\{
u \in L^{p}(\Omega) :
\nabla_H u \in L^{p}(\Omega)
\right\},
\]
is defined via the horizontal gradient
\[
\nabla_H u = (X_1 u, X_2 u, \ldots, X_{2N} u),
\]
and equipped with the norm
\[
\|u\|_{HW^{1,p}(\Omega)}
=
\|u\|_{L^{p}(\Omega)}
+
\|\nabla_H u\|_{L^{p}(\Omega)}.
\]
With this norm, $HW^{1,p}(\Omega)$ becomes a Banach space.
The local Sobolev space $HW_{\mathrm{loc}}^{1,p}(\Omega)$ is defined by
\[
HW_{\mathrm{loc}}^{1,p}(\Omega)
=
\left\{
u : u \in HW^{1,p}(\Omega') \ \text{for every open set } \Omega' \Subset \Omega
\right\}.
\]

Let $1 \le p < \infty$, $s \in (0,1)$, and let $v:\mathbb{H}^N \to \mathbb{R}$ be a measurable function.  
The associated Gagliardo seminorm is given by
\[
[v]_{HW^{s,p}(\mathbb{H}^N)}
=
\left(
\int_{\mathbb{H}^N}\int_{\mathbb{H}^N}
\frac{|v(x)-v(y)|^{p}}
{|y^{-1}\circ x|^{Q+sp}}
\, dx\, dy
\right)^{1/p}.
\]

The fractional Sobolev space on the Heisenberg group is defined as
\[
HW^{s,p}(\mathbb{H}^N)
=
\left\{
v \in L^{p}(\mathbb{H}^N) :
[v]_{HW^{s,p}(\mathbb{H}^N)} < \infty
\right\},
\]
and it is endowed with the natural norm
\[
\|v\|_{HW^{s,p}(\mathbb{H}^N)}
=
\left(
\|v\|_{L^{p}(\mathbb{H}^N)}^{p}
+
[v]_{HW^{s,p}(\mathbb{H}^N)}^{p}
\right)^{1/p}.
\]

For any open set $\Omega \subset \mathbb{H}^N$, the fractional Sobolev space $HW^{s,p}(\Omega)$ and its associated norm 
$\|v\|_{HW^{s,p}(\Omega)}$ are defined analogously to the whole-space case.

To study the mixed problem~\eqref{meqn} with $\alpha>0$, we consider the space
\[
HW_0^{1,p}(\Omega) = \{ u \in HW^{1,p}(\mathbb{H}^N) : u = 0 \ \text{in } \mathbb{H}^N \setminus \Omega \}.
\]

The following result is a direct consequence of \cite[Theorem 2.4]{DC}.

\begin{Lemma}\label{l2.3}
Let $\Omega$ be a bounded domain in $\mathbb{H}^N$, and let $1 < p < \infty$ and $0 < s < 1$.  
Then there exists a constant $C = C(Q,p,s,\Omega) > 0$ such that
\[
\iint_{\mathbb{H}^N \times \mathbb{H}^N}
\frac{|u(x)-u(y)|^{p}}{|x-y|^{Q+ps}}\, dx\, dy
\le
C \int_{\Omega} |\nabla_H u(x)|^{p}\, dx,
\]
for every $u \in HW^{1,p}_0(\Omega)$.
\end{Lemma}

For the following embedding result, we refer to \cite[Theorem 8.1]{Koskela}.

\begin{Lemma}\label{emb}
Let $0 < s < 1$ and $1 < p < Q$. Then the embedding
\[
HW^{1,p}(\Omega) \hookrightarrow L^r(\Omega)
\]
is continuous for every $r \in [1,p^*]$ and compact for every $r \in [1,p^*)$.
\end{Lemma}

Taking into account Lemmas~\ref{l2.3} and~\ref{emb}, for a given $\alpha>0$, we introduce the following norm $\|\cdot\|$ on $HW_0^{1,p}(\Omega)$, which will be used throughout the rest of the paper:
\begin{equation}\label{eqnorm}
\|u\|_{HW_0^{1,p}(\Omega)}
:= \left(
\int_{\Omega} |\nabla_H u|^p \, dx
+ \alpha \int_{\mathbb{H}^N} \int_{\mathbb{H}^N} 
\frac{|u(x)-u(y)|^p}{|y^{-1}\circ x|^{Q+sp}} \, dx\, dy
\right)^{1/p}.
\end{equation}

\subsection{Notations and assumptions}
We fix the following notations and assumptions, which will be used throughout the paper unless otherwise specified:

\begin{itemize}
    \item For $l > 1$, the conjugate exponent is denoted by $l' = \frac{l}{l-1}$.

    \item We assume $\alpha>0$, $0 < s < 1 < p < Q$, and $\delta \in (0,1)$, where $Q = 2N+2$ is the homogeneous dimension of $\mathbb{H}^N$. The critical Sobolev exponent is
    \[
    p^* = \frac{Qp}{Q - p}, \qquad 1 < p < Q.
    \]

    \item For $q > 1$, we define
    \[
    J_q(t) = |t|^{q-2} t \quad \text{for all } t \in \mathbb{R}.
    \]

    \item We use the notation
    \[
    d\mu = \frac{dx\, dy}{|y^{-1} \circ x|^{Q + sp}}.
    \]

    \item $\Omega$ denotes a bounded smooth domain in $\mathbb{H}^N$.

    \item For $\omega \Subset \Omega$, we mean that $\omega$ is compactly contained in $\Omega$, i.e.,
    \[
    \omega \subset \overline{\omega} \subset \Omega.
    \]

    \item For $k \in \mathbb{R}$, we adopt the standard notation
    \[
    k^+ = \max\{k,0\}, \qquad 
    k^- = \max\{-k,0\}, \qquad 
    k_- = \min\{k,0\}.
    \]

    \item For a measurable function $F$ on a set $S$ and constants $c,d$, the notation $c \le F \le d$ in $S$ means
    \[
    c \le F(x) \le d \quad \text{for almost every } x \in S.
    \]

    \item The symbol $C$ denotes a generic positive constant, which may change from line to line. If the constant depends on parameters $r_1, r_2, \ldots, r_k$, we write
    \[
    C = C(r_1, r_2, \ldots, r_k).
    \]
\end{itemize}

Before stating our main results, we mention some auxiliary results below.
\subsection{Auxiliary results}
The following result is a direct consequence of \cite[Lemma 4.4]{G1}.

\begin{Lemma}\label{auxresult}
Let $g \in L^{\infty}(\Omega) \setminus \{0\}$ be a nonnegative function in $\Omega$.  
Then there exists a unique solution 
\[
u \in HW_0^{1,p}(\Omega) \cap L^{\infty}(\Omega)
\] 
to the problem
\begin{equation}\label{approxnew}
M_\alpha u = g \quad \text{in } \Omega, \qquad
u > 0 \ \text{in } \Omega, \qquad
u = 0 \ \text{in } \mathbb{H}^N \setminus \Omega.
\end{equation}
Moreover, for every $\omega \Subset \Omega$, there exists a constant $C(\omega) > 0$ such that 
\[
u \ge C(\omega) \quad \text{in } \omega.
\]
\end{Lemma}

The next result follows from \cite[Lemma 3.1 and Lemma 3.6]{G1}.

\begin{Lemma}\label{approx}
Let $f \in L^{1}(\Omega) \setminus \{0\}$ be a nonnegative function, and for $n \in \mathbb{N}$, define
\[
f_n(x) := \min\{f(x), n\}.
\]
Consider the approximated problem
\begin{equation}\label{approxeqn}
M_\alpha u = f_n(x) \Big(u^{+} + \frac{1}{n}\Big)^{-\delta} \quad \text{in } \Omega, \qquad
u = 0 \quad \text{in } \mathbb{H}^N \setminus \Omega.
\end{equation}
Then, for each $n \in \mathbb{N}$, there exists a unique positive solution 
\[
u_n \in HW_0^{1,p}(\Omega) \cap L^{\infty}(\Omega).
\] 
Furthermore, the sequence $\{u_n\}_{n\in\mathbb{N}}$ is monotone increasing, i.e.,
\[
u_{n+1} \ge u_n \quad \text{in } \Omega \text{ for all } n \in \mathbb{N}.
\]
In addition, for every compact set $\omega \Subset \Omega$, there exists a constant $C(\omega) > 0$, independent of $n$, such that
\[
u_n \ge C(\omega) > 0 \quad \text{in } \omega.
\]
Moreover, if $f\in L^m(\Om)$, where \[
m = \left( \frac{p^*}{1-\delta} \right)',
\] then the sequence $\{u_n\}_{n\in\mathbb{N}}$ is uniformly bounded in $HW_0^{1,p}(\Omega)$.
\end{Lemma}

Before presenting the next result, we introduce the notion of weak solutions for problem~\eqref{meqn}.  

\begin{Definition}[Weak Solution]\label{wksoldef}
Let $f \in L^1(\Omega) \setminus \{0\}$ be a nonnegative function. A function 
\[
u \in HW_0^{1,p}(\Omega)
\] 
is called a \emph{weak solution} of problem~\eqref{meqn} if for every subset $\omega \Subset \Omega$, there exists a constant $C = C(\omega) > 0$ such that $u \ge C$ in $\omega$, and for every test function $\varphi \in C_c^1(\Omega)$, it holds that
\begin{equation}\label{wksoleqn}
\int_{\Omega} |\nabla_H u|^{p-2} \nabla_H u \cdot \nabla_H \varphi \, dx
+ \alpha \int_{\mathbb{H}^N}\int_{\mathbb{H}^N} J_p(u(x)-u(y)) \, d\mu
= \int_{\Omega} f(x)\, u^{-\delta} \varphi(x) \, dx.
\end{equation}
\end{Definition}

\begin{Remark}\label{Rmwksol}
By density arguments (see \cite[Lemma 5.1]{GU21}), the identity~\eqref{wksoleqn} extends to all $\varphi \in HW_0^{1,p}(\Omega)$. Consequently, following the proof of \cite[Corollary 5.2]{GU21}, one concludes that problem~\eqref{meqn} admits at most one weak solution in $HW_0^{1,p}(\Omega)$ whenever $f \in L^1(\Omega) \setminus \{0\}$.
\end{Remark}

The next result can be found in \cite[Theorem 1.2 and Theorem 1.7]{G1}.

\begin{Theorem}\label{thm1}
Let $f \in L^m(\Omega) \setminus \{0\}$ be nonnegative, where 
\[
m = \left( \frac{p^*}{1-\delta} \right)'.
\] 
Then problem~\eqref{meqn} admits a unique weak solution 
\[
u_\delta \in HW_0^{1,p}(\Omega).
\]
\end{Theorem}

We present our main results in the following subsection. The remainder of the paper is organized as follows: In Section 2, we discuss preliminary results, and in Section 3, we prove the main theorems.

\subsection{Main results}

\begin{Theorem}\label{thm5}
Let $f \in L^m(\Omega) \setminus \{0\}$ be nonnegative, where 
\[
m = \Big(\frac{p^*}{1-\delta}\Big)'.
\] 
Let $u_\delta \in HW_0^{1,p}(\Omega)$ denote the weak solution of problem~\eqref{meqn} given by Theorem~\ref{thm1}. Then the following hold:

\begin{enumerate}
\item[(a)] \textbf{(Extremal)}  
\begin{align*}
\Theta(\Omega) = \left(
\int_{\Omega} |\nabla_H u_\delta|^p \, dx
+ \int_{\mathbb{H}^N} \int_{\mathbb{H}^N} |u_\delta(x) - u_\delta(y)|^p \, d\mu
\right)^{\frac{1-\delta-p}{1-\delta}},
\end{align*}
where $\Theta(\Omega)$ is defined in \eqref{Theta}.

\item[(b)] \textbf{(Sobolev-type inequality)}  
For every $v \in HW_0^{1,p}(\Omega)$, the following mixed Sobolev inequality holds:
\begin{equation}\label{inequality2}
C \left( \int_\Omega |v|^{1-\delta} f \, dx \right)^{\frac{p}{1-\delta}}
\le 
\int_{\Omega} |\nabla_H v|^p \, dx
+ \int_{\mathbb{H}^N} \int_{\mathbb{H}^N} |v(x) - v(y)|^p \, d\mu,
\end{equation}
if and only if 
\[
C \le \Theta(\Omega).
\]

\item[(c)] \textbf{(Simplicity)}  
If for some $w \in HW_0^{1,p}(\Omega)$ the equality
\begin{equation}\label{sim}
\Theta(\Omega) \left( \int_\Omega |w|^{1-\delta} f \, dx \right)^{\frac{p}{1-\delta}}
= 
\int_{\Omega} |\nabla_H w|^p \, dx
+ \int_{\mathbb{H}^N} \int_{\mathbb{H}^N} |w(x) - w(y)|^p \, d\mu
\end{equation}
holds, then $w = k u_\delta$ for some constant $k$.
\end{enumerate}
\end{Theorem}

\begin{Corollary}\label{ansiowgtrmk}
Define
\[
S_{\delta} := 
\left\{
v \in HW^{1,p}_0(\Omega) : 
\int_\Omega |v|^{1-\delta} f \, dx = 1
\right\},
\qquad
\zeta_{\delta} := 
\left( \int_\Omega u_\delta^{1-\delta} f \, dx \right)^{-\frac{1}{1-\delta}}.
\]
Then 
\[
V_\delta := \zeta_\delta u_\delta \in S_\delta,
\] 
and from Theorem~\ref{thm5}, we have
\[
\Theta(\Omega) = 
\int_{\Omega} |\nabla_H V_\delta|^p \, dx
+ \int_{\mathbb{H}^N} \int_{\mathbb{H}^N} |V_\delta(x) - V_\delta(y)|^p \, d\mu.
\]
Moreover, $V_\delta$ satisfies the singular problem
\[
M_\alpha V_\delta = \Theta(\Omega) f V_\delta^{-\delta} \quad \text{in } \Omega, \qquad V_\delta > 0 \ \text{in } \Omega.
\]
\end{Corollary}

\begin{Theorem}\label{Theorem 2.21}
Let $f \in L^1(\Omega) \setminus \{0\}$ be nonnegative.  
Then, for every $v \in HW_0^{1,p}(\Omega)$, the mixed Sobolev inequality~\eqref{inequality2} holds 
if and only if the problem~\eqref{meqn} admits a weak solution in $HW_0^{1,p}(\Omega)$.
\end{Theorem}

The following results concern the purely local and purely nonlocal cases.

\begin{Theorem}\label{Theorem 2.22}
Let $f \in L^1(\Omega) \setminus \{0\}$ be nonnegative.  
Then, for every $v \in HW_0^{1,p}(\Omega)$, the local Sobolev inequality
\begin{equation}\label{localSob}
C \left( \int_\Omega |v|^{1-\delta} f \, dx \right)^{\frac{p}{1-\delta}}
\le \int_\Omega |\nabla_H v|^p \, dx,
\end{equation}
holds if and only if the singular $p$-Laplace equation~\eqref{plh} admits a weak solution in the usual Sobolev space $HW_0^{1,p}(\Omega)$.
\end{Theorem}

\begin{Theorem}\label{Theorem 2.23}
Let $f \in L^1(\Omega) \setminus \{0\}$ be nonnegative.  
Then, for every $v \in HW_0^{s,p}(\Omega)$, the nonlocal Sobolev inequality
\begin{equation}\label{nonlocalSob}
C \left( \int_\Omega |v|^{1-\delta} f \, dx \right)^{\frac{p}{1-\delta}}
\le \int_{\mathbb{H}^N} \int_{\mathbb{H}^N} |v(x)-v(y)|^p \, d\mu,
\end{equation}
holds if and only if the singular fractional $p$-Laplace equation
\begin{equation}\label{pfrach}
(-\Delta_p)^s u = f(x) u^{-\delta} \quad \text{in } \Omega, \qquad
u > 0 \text{ in } \Omega, \qquad
u = 0 \text{ in } \mathbb{H}^N \setminus \Omega
\end{equation}
admits a weak solution in $HW_0^{s,p}(\Omega)$.
\end{Theorem}

\begin{Remark}\label{rmk1}
The weak solutions of problems~\eqref{plh} and~\eqref{pfrach} are defined analogously to 
\cite[Definition 3.1]{G1} and \cite[Definition 2.3]{GarainHein}, respectively.
\end{Remark}

As a direct consequence of Theorems~\ref{Theorem 2.21},~\ref{Theorem 2.22} and Lemma~\ref{l2.3}, we obtain the following result connecting the problems~\eqref{meqn} and~\eqref{plh}.

\begin{Corollary}\label{Corollary 2.25}
Let $f \in L^1(\Omega) \setminus \{0\}$ be nonnegative.  
Then problem~\eqref{meqn} admits a weak solution in $HW_0^{1,p}(\Omega)$ 
if and only if problem~\eqref{plh} admits a weak solution in the usual Sobolev space $HW_0^{1,p}(\Omega)$.
\end{Corollary}

\section{Preliminaries for the Sobolev-Type Inequality with Extremal}

Throughout this section, we assume $f \in L^m(\Omega) \setminus \{0\}$, where 
\[
m = \left( \frac{p^*}{1 - \delta} \right)',
\] 
unless stated otherwise. Let $u_n \in HW_0^{1,p}(\Omega)$ denote the solution of the approximated problem~\eqref{approxeqn} provided by Lemma~\ref{approx}. By the proof of Theorem~\ref{thm1}, the sequence $\{u_n\}_{n\in\mathbb{N}}$ converges pointwise to a function $u_\delta \in HW_0^{1,p}(\Omega)$, which is the weak solution of problem~\eqref{meqn}, satisfying $u_n \le u_\delta$ in $\Omega$ for all $n \in \mathbb{N}$. In the following, we establish several auxiliary results that will be useful for proving the main results.

\begin{Lemma}\label{lemma1}
Let $n \in \mathbb{N}$. Then, for every $\phi \in HW_0^{1,p}(\Omega)$, the following inequality holds:
\begin{equation}\label{prop1}
\|u_n\|^p \le \|\phi\|^p + p \int_\Omega (u_n - \phi) \, \Big(u_n + \frac{1}{n}\Big)^{-\delta} f_n \, dx.
\end{equation}
Moreover, the sequence $\{\|u_n\|\}_{n \in \mathbb{N}}$ is nondecreasing, i.e.,
\begin{equation}\label{nmon}
\|u_n\| \le \|u_{n+1}\| \quad \text{for every } n \in \mathbb{N}.
\end{equation}
\end{Lemma}

\begin{proof}
Let $h \in HW_0^{1,p}(\Omega)$. By Lemma~\ref{auxresult}, there exists a unique solution 
\[
v \in HW_0^{1,p}(\Omega)
\] 
to the problem
\[
M_\alpha v = f_n(x) \Big(h^+ + \frac{1}{n}\Big)^{-\delta}, \quad v > 0 \text{ in } \Omega, \quad v = 0 \text{ in } \mathbb{H}^N \setminus \Omega.
\] 
Furthermore, $v$ minimizes the functional $J: HW_0^{1,p}(\Omega) \to \mathbb{R}$ defined by
\[
J(\phi) := \frac{1}{p} \|\phi\|^p - \int_\Omega f_n \Big(h^+ + \frac{1}{n}\Big)^{-\delta} \phi \, dx.
\]

Hence, for every $\phi \in HW_0^{1,p}(\Omega)$, we have $J(v) \le J(\phi)$, which gives
\begin{equation}\label{mineqn}
\frac{1}{p} \|v\|^p - \int_\Omega \frac{f_n}{(h^+ + \frac{1}{n})^\delta} v \, dx
\le 
\frac{1}{p} \|\phi\|^p - \int_\Omega f_n \Big(h^+ + \frac{1}{n}\Big)^{-\delta} \phi \, dx.
\end{equation}

Inequality~\eqref{prop1} follows by taking $v = h = u_n$ in \eqref{mineqn}.  
Next, choosing $\phi = u_{n+1}$ in \eqref{prop1} and using the monotonicity property $u_n \le u_{n+1}$ from Lemma~\ref{approx}, we deduce \eqref{nmon}.
\end{proof}

\begin{Lemma}\label{strong}
Up to a subsequence, the sequence $\{u_n\}_{n\in\mathbb{N}}$ converges strongly to $u_\delta$ in $HW_0^{1,p}(\Omega)$.
\end{Lemma}

\begin{proof}
Note that $u_n \le u_\delta$. By taking $\phi = u_\delta$ in \eqref{prop1}, we obtain
\[
\|u_n\| \le \|u_\delta\|,
\]
which, together with the monotonicity property \eqref{nmon} from Lemma~\ref{lemma1}, gives
\begin{equation}\label{lim1}
\lim_{n \to \infty} \|u_n\| \le \|u_\delta\|.
\end{equation}

By Lemma~\ref{approx}, the sequence $\{u_n\}_{n\in\mathbb{N}}$ is uniformly bounded in $HW_0^{1,p}(\Omega)$, so up to a subsequence, $u_n \rightharpoonup u_\delta$ weakly in $HW_0^{1,p}(\Omega)$. Consequently,
\begin{equation}\label{lim2}
\|u_\delta\| \le \lim_{n \to \infty} \|u_n\|.
\end{equation}

Combining \eqref{lim1} and \eqref{lim2} and using the uniform convexity of $HW_0^{1,p}(\Omega)$, the strong convergence follows.
\end{proof}

\begin{Lemma}\label{minprop}
Define the functional $I_{\delta} : HW_0^{1,p}(\Omega) \to \mathbb{R}$ by
\[
I_{\delta}(v) := \frac{1}{p} \|v\|^p - \frac{1}{1-\delta} \int_\Omega (v^+)^{1-\delta} f \, dx.
\]
Then $u_\delta$ is a minimizer of $I_\delta$.
\end{Lemma}

\begin{proof}
Consider the auxiliary functional $I_n : HW_0^{1,p}(\Omega) \to \mathbb{R}$ given by
\[
I_n(v) := \frac{1}{p} \|v\|^p - \int_\Omega G_n(v) f_n \, dx,
\]
where
\[
G_n(t) := \frac{1}{1-\delta} \left(t^+ + \frac{1}{n}\right)^{1-\delta} - \left(\frac{1}{n}\right)^{-\delta} t^-.
\]
Note that $I_n$ is coercive, bounded from below, and of class $C^1$, so it admits a minimizer $v_n \in HW_0^{1,p}(\Omega)$ such that
\[
\langle I_n'(v_n), \phi \rangle = 0 \quad \text{for all } \phi \in HW_0^{1,p}(\Omega).
\]

Since $I_n(v_n) \le I_n(v_n^+)$, it follows that $v_n \ge 0$ in $\Omega$. Hence, $v_n$ solves the approximated problem~\eqref{approxeqn}. By the uniqueness result in Lemma~\ref{approx}, we have $u_n = v_n$, showing that $u_n$ is a minimizer of $I_n$. Therefore, for all $v \in HW_0^{1,p}(\Omega)$,
\begin{equation}\label{min}
I_n(u_n) \le I_n(v^+).
\end{equation}

Since $u_n \le u_\delta$, the Lebesgue dominated convergence theorem gives
\[
\lim_{n \to \infty} \int_\Omega G_n(u_n) f_n \, dx = \frac{1}{1-\delta} \int_\Omega u_\delta^{1-\delta} f \, dx.
\]
Moreover, by Lemma~\ref{strong},
\[
\lim_{n \to \infty} \|u_n\| = \|u_\delta\|,
\]
so that
\begin{equation}\label{newlim2}
\lim_{n \to \infty} I_n(u_n) = I_\delta(u_\delta).
\end{equation}

Additionally, for any $v \in HW_0^{1,p}(\Omega)$,
\begin{equation}\label{lim3}
\lim_{n \to \infty} \int_\Omega G_n(v^+) f_n \, dx = \frac{1}{1-\delta} \int_\Omega (v^+)^{1-\delta} f \, dx.
\end{equation}

Since $\|v^+\| \le \|v\|$, using \eqref{newlim2} and \eqref{lim3} in \eqref{min}, we conclude that
\[
I_\delta(u_\delta) \le I_\delta(v) \quad \text{for all } v \in HW_0^{1,p}(\Omega).
\]
This completes the proof.
\end{proof}

\begin{Lemma}\label{l3.5}
Suppose $f \in L^{1}(\Omega)\setminus\{0\}$ is nonnegative. Then, the sequence $\{u_n\}$, where $u_n$, $n \in \mathbb{N}$, are solutions of the approximate problem \eqref{approxeqn}, is uniformly bounded in $HW^{1,p}_0(\Omega)$, provided that \eqref{inequality2} holds.
\end{Lemma}

\begin{proof}
Taking $u_n$ as a test function in the weak formulation of \eqref{approxeqn} and using inequality \eqref{inequality2} along with Lemma~\ref{l2.3}, we obtain
\[
\|u_n\|_p^p 
\le 
\int_{\Omega} f_n(x)\, u_n \left( u_n + \frac{1}{n} \right)^{-\delta} dx
\le
\int_{\Omega} u_n^{\,1-\delta} f \, dx
\le
C \|u_n\|^{\,1-\delta},
\]
for some constant $C>0$ independent of $n$. This gives the desired uniform bound.
\end{proof}

\section{Proof of the Sobolev Type Inequality with Extremal}

\textbf{Proof of Theorem \ref{thm5}:}  
Recall the norm on $HW_0^{1,p}(\Omega)$ defined by  
\[
\|v\| = \left(\int_{\Omega}|\nabla_H v|^p\,dx + \alpha \int_{\mathbb{H}^N} \int_{\mathbb{H}^N} J_p(v(x) - v(y))\,d\mu\right)^{\frac{1}{p}}.
\]

\begin{enumerate}
\item[(a)] To prove part (a), it suffices to show  
\[
\Theta(\Omega) := \inf_{v \in S_\delta} \|v\|^p = \|u_\delta\|^{\frac{p(1-\delta-p)}{1-\delta}},
\]
where  
\[
S_\delta = \left\{ v \in HW_0^{1,p}(\Omega) : \int_\Omega |v|^{1-\delta} f \, dx = 1 \right\}.
\]

Define  
\[
V_\delta = \tau_\delta u_\delta, \quad \text{with} \quad \tau_\delta = \left(\int_\Omega u_\delta^{1-\delta} f \, dx\right)^{-\frac{1}{1-\delta}},
\]
so that $V_\delta \in S_\delta$.  

By Remark~\ref{Rmwksol}, testing \eqref{wksoleqn} with $u_\delta$ yields  
\[
\|u_\delta\|^p = \int_\Omega u_\delta^{1-\delta} f \, dx,
\]
and therefore
\[
\|V_\delta\|^p = \tau_\delta^p \|u_\delta\|^p = \left( \int_\Omega u_\delta^{1-\delta} f \, dx \right)^{-\frac{p}{1-\delta}} \|u_\delta\|^p = \|u_\delta\|^{\frac{p(1-\delta-p)}{1-\delta}}.
\]

For any $v \in S_\delta$, set  
\[
\lambda = \|v\|^{-\frac{p}{p+\delta-1}}.
\]
Since $u_\delta$ minimizes $I_\delta$, we have  
\[
I_\delta(u_\delta) \le I_\delta(\lambda |v|),
\]
which implies  
\[
\|u_\delta\|^{\frac{p(1-\delta-p)}{1-\delta}} \le \|v\|^p.
\]
Taking the infimum over $v \in S_\delta$ gives  
\[
\Theta(\Omega) = \|u_\delta\|^{\frac{p(1-\delta-p)}{1-\delta}}.
\]

\item[(b)] Suppose the mixed Sobolev inequality \eqref{inequality2} holds.  

If $C > \Theta(\Omega)$, then by part (a),  
\[
C \left( \int_\Omega V_\delta^{1-\delta} f \, dx \right)^{\frac{p}{1-\delta}} > \|V_\delta\|^p,
\]
contradicting \eqref{inequality2}.  

Conversely, if $C \le \Theta(\Omega)$, for any $v \in HW_0^{s,p}(\Omega) \setminus \{0\}$, define  
\[
w = \left( \int_\Omega |v|^{1-\delta} f \, dx \right)^{-\frac{1}{1-\delta}} v \in S_\delta.
\]
Then  
\[
C \le \|w\|^p = \left( \int_\Omega |v|^{1-\delta} f \, dx \right)^{-\frac{p}{1-\delta}} \|v\|^p,
\]
which proves \eqref{inequality2}.

\item[(c)] From part (a), $\Theta(\Omega) = \|V_\delta\|^p$.  

Suppose $v \in S_\delta$ also attains this minimum, i.e., $\|v\|^p = \Theta(\Omega)$.  

First, $v$ cannot change sign in $\Omega$, because if it did,  
\[
\||v|\|^p < \|v\|^p,
\]
contradicting the minimality since $|v| \in S_\delta$.  

Without loss of generality, assume $v \ge 0$ and define  
\[
g = \left( \int_\Omega \left( \frac{v}{2} + \frac{V_\delta}{2} \right)^{1-\delta} f \, dx \right)^{\frac{1}{1-\delta}}, \quad
h = \frac{v + V_\delta}{2g} \in S_\delta.
\]

By strict convexity,  
\[
\Theta(\Omega) \le \|h\|^p \le \frac{1}{g^p} \left\| \frac{v + V_\delta}{2} \right\|^p \le \frac{\Theta(\Omega)}{g^p} \le \Theta(\Omega),
\]
forcing $g = 1$ and equality in the convexity inequality. Hence  
\[
v = V_\delta.
\]

Thus, the minimizer is unique up to sign.  

If \eqref{sim} holds for some nonzero $w$, then  
\[
\gamma w \in S_\delta, \quad \gamma = \left( \int_\Omega |w|^{1-\delta} f \, dx \right)^{-\frac{1}{1-\delta}},
\]
which implies  
\[
w = \pm \gamma^{-1} V_\delta = \pm \gamma^{-1} \tau_\delta u_\delta.
\]
\qed
\end{enumerate}
\textbf{Proof of Theorem \ref{Theorem 2.21}:}  
Assume that the inequality \eqref{inequality2} holds. Then, by Lemma \ref{l3.5}, the sequence $\{u_n\}_{n\in\mathbb{N}}$ is uniformly bounded in $HW^{1,p}_{0}(\Omega)$.  
Following the same arguments as in the proof of \cite[Theorem 1.2]{G1}, it follows that the problem \eqref{meqn} admits a weak solution $u \in HW^{1,p}_{0}(\Omega)$.

Conversely, let $u \in HW^{1,p}_{0}(\Omega)$ be a weak solution of \eqref{meqn}.  
By Remark \ref{Rmwksol}, choosing $u$ as a test function in \eqref{meqn} gives
\begin{equation}\label{eq6.11}
\|u\|^{p} = \int_{\Omega} u^{1-\delta} f \, dx. 
\end{equation}

Next, again using Remark \ref{Rmwksol}, take $|v| \in HW^{1,p}_{0}(\Omega)$ as a test function in \eqref{meqn}.  
Applying H\"older's inequality, we obtain
\begin{equation}\label{eq6.12}
\int_{\Omega} |v|\, u^{-\delta} f \, dx \le C \|u\|^{p-1} \|v\|,
\end{equation}
for some constant $C>0$.  

Combining \eqref{eq6.11} and \eqref{eq6.12}, for any $v \in HW^{1,p}_{0}(\Omega)$, we have
\[
\begin{aligned}
\int_{\Omega} |v|^{1-\delta} f \, dx 
&= \int_{\Omega} \big( |v| u^{-\delta} f \big)^{1-\delta} \big( u^{1-\delta} f \big)^{\delta} dx \\
&\le \left( \int_{\Omega} |v| u^{-\delta} f \, dx \right)^{1-\delta} \left( \int_{\Omega} u^{1-\delta} f \, dx \right)^{\delta} \\
&\le C \|u\|^{p+\delta-1} \|v\|^{1-\delta},
\end{aligned}
\]
which establishes the inequality \eqref{inequality2}. \qed

\medskip

\textit{Proof of Theorem \ref{Theorem 2.22}:}  
Using \cite[Lemma 4.2]{GUaamp} and proceeding as in the proof of Theorem \ref{Theorem 2.21}, the result follows. \(\square\)

\medskip

\textit{Proof of Theorem \ref{Theorem 2.23}:}  
Using \cite[Lemma 3.1]{GarainHein} and following the same arguments as in the proof of Theorem \ref{Theorem 2.21}, the result follows. \(\square\)

\section*{Acknowledgment}
This work is supported ANRF Research Grant, File No. ANRF/ECRG/2024/000780/PMS.

\noindent {\textsf{Prashanta Garain\\Department of Mathematical Sciences\\
Indian Institute of Science Education and Research Berhampur\\ Permanent Campus, At/Po:-Laudigam, Dist.-Ganjam\\
Odisha, India-760003
}\\ 
\textsf{e-mail}: pgarain92@gmail.com\\


\begin{thebibliography}{10}

\bibitem{GFA}
Giovanni Anello, Francesca Faraci, and Antonio Iannizzotto.
\newblock On a problem of {H}uang concerning best constants in {S}obolev embeddings.
\newblock {\em Ann. Mat. Pura Appl. (4)}, 194(3):767--779, 2015.

\bibitem{ARad}
Rakesh Arora and Vicen\c tiu~D. R\u~adulescu.
\newblock Combined effects in mixed local-nonlocal stationary problems.
\newblock {\em Proc. Roy. Soc. Edinburgh Sect. A}, 155(1):10--56, 2025.

\bibitem{BGmm22}
Kaushik Bal and Prashanta Garain.
\newblock Weighted and anisotropic {S}obolev inequality with extremal.
\newblock {\em Manuscripta Math.}, 168(1-2):101--117, 2022.

\bibitem{BGdie}
Kaushik Bal and Prashanta Garain.
\newblock Weighted anisotropic {S}obolev inequality with extremal and associated singular problems.
\newblock {\em Differential Integral Equations}, 36(1-2):59--92, 2023.

\bibitem{Nsop}
Bego\~na Barrios, Ida De~Bonis, Mar\'ia Medina, and Ireneo Peral.
\newblock Semilinear problems for the fractional laplacian with a singular nonlinearity.
\newblock {\em Open Math.}, 13(1):390--407, 2015.

\bibitem{Vecchi}
Stefano Biagi and Eugenio Vecchi.
\newblock Multiplicity of positive solutions for mixed local-nonlocal singular critical problems.
\newblock {\em Calc. Var. Partial Differential Equations}, 63(9):Paper No. 221, 45, 2024.

\bibitem{Boc-Or}
Lucio Boccardo and Luigi Orsina.
\newblock Semilinear elliptic equations with singular nonlinearities.
\newblock {\em Calc. Var. Partial Differential Equations}, 37(3-4):363--380, 2010.

\bibitem{Caninoetal}
Annamaria Canino, Luigi Montoro, Berardino Sciunzi, and Marco Squassina.
\newblock Nonlocal problems with singular nonlinearity.
\newblock {\em Bull. Sci. Math.}, 141(3):223--250, 2017.

\bibitem{Canino}
Annamaria Canino, Berardino Sciunzi, and Alessandro Trombetta.
\newblock Existence and uniqueness for {$p$}-{L}aplace equations involving singular nonlinearities.
\newblock {\em NoDEA Nonlinear Differential Equations Appl.}, 23(2):Art. 8, 18, 2016.

\bibitem{DC}
Debajyoti Choudhuri, Leandro~S. Tavares, and Du\v san~D. Repov\v~s.
\newblock A multiphase eigenvalue problem on a stratified {L}ie group.
\newblock {\em Rend. Circ. Mat. Palermo (2)}, 73(7):2533--2546, 2024.

\bibitem{CRT}
M.~G. Crandall, P.~H. Rabinowitz, and L.~Tartar.
\newblock On a {D}irichlet problem with a singular nonlinearity.
\newblock {\em Comm. Partial Differential Equations}, 2(2):193--222, 1977.

\bibitem{Nmn}
G.~Ercole and G.~A. Pereira.
\newblock Fractional {S}obolev inequalities associated with singular problems.
\newblock {\em Math. Nachr.}, 291(11-12):1666--1685, 2018.

\bibitem{EP}
Grey Ercole and Gilberto de~Assis Pereira.
\newblock On a singular minimizing problem.
\newblock {\em J. Anal. Math.}, 135(2):575--598, 2018.

\bibitem{Garainjga}
Prashanta Garain.
\newblock On a class of mixed local and nonlocal semilinear elliptic equation with singular nonlinearity.
\newblock {\em J. Geom. Anal.}, 33(7):Paper No. 212, 20, 2023.

\bibitem{Gjms}
Prashanta Garain.
\newblock Mixed anisotropic and nonlocal {S}obolev type inequalities with extremal.
\newblock {\em J. Math. Sci. (N.Y.)}, 281(5):633--645, 2024.

\bibitem{G1}
Prashanta Garain.
\newblock Mixed local-nonlocal $p$-laplace equation with variable singular nonlinearity in the heisenberg group, 2025.

\bibitem{GarainHein}
Prashanta Garain.
\newblock Nonlocal singular problem and associated sobolev type inequality with extremal in the heisenberg group, 2025.

\bibitem{GU21}
Prashanta Garain and Alexander Ukhlov.
\newblock Mixed local and nonlocal {S}obolev inequalities with extremal and associated quasilinear singular elliptic problems.
\newblock {\em Nonlinear Anal.}, 223:Paper No. 113022, 35, 2022.

\bibitem{GUaamp}
Prashanta Garain and Alexander Ukhlov.
\newblock Singular subelliptic equations and {S}obolev inequalities on {C}arnot groups.
\newblock {\em Anal. Math. Phys.}, 12(2):Paper No. 67, 18, 2022.

\bibitem{Koskela}
Piotr Haj\l~asz and Pekka Koskela.
\newblock Sobolev met {P}oincar\'e.
\newblock {\em Mem. Amer. Math. Soc.}, 145(688):x+101, 2000.

\bibitem{Zhang}
Junli Zhang and Pengcheng Niu.
\newblock Regularity for mixed local and nonlocal degenerate elliptic equations in the {H}eisenberg group.
\newblock {\em J. Differential Equations}, 453:113888, 2026.

\bibitem{Zhang2}
Junli Zhang, Pengcheng Niu, and Xuewen Wu.
\newblock Higher weak differentiability to mixed local and nonlocal degenerate elliptic equations in the heisenberg group.
\newblock 12 2024.

\end{thebibliography}
\end{document}